\documentclass[12pt]{amsart}
\usepackage{amscd, amssymb, amsmath}
\input{xy}
\xyoption{all}

\newtheorem{Thm}{Theorem}[section]

\newtheorem{Prop}[Thm]{Proposition}
\newtheorem{Def/Thm}[Thm]{Definition/Theorem}
\newtheorem{Cor}[Thm]{Corollary}
\newtheorem{Lemma}[Thm]{Lemma}
\theoremstyle{remark}
\newtheorem{Def}[Thm]{Definition}

\numberwithin{equation}{section}

\newcommand{\ot }{\otimes}
\newcommand{\ra }{\rightarrow}

\newcommand{\Ext}{{\mathrm{Ext}}}

\newcommand{\Hom }{{\mathrm{Hom}}}

\newcommand{\Ker}{{\mathrm{Ker}}}

\newcommand{\rank }{{\mathrm{rank}}}
\newcommand{\cA}{{\mathcal{A}}}
\newcommand{\cO}{{\mathcal{O}}}

\newcommand{\cH}{{\mathcal{H}}}

\newcommand{\NN}{{\mathbb N}}

\newcommand{\PP }{{\mathbb P}}

\newcommand{\QQ }{{\mathbb Q}}
\newcommand{\CC }{{\mathbb C}}
\newcommand{\ZZ }{{\mathbb Z}}
\newcommand{\RR }{{\mathbb R}}

\newcommand{\ke }{{\varepsilon }}

\newcommand{\ka }{{\alpha}}

\newcommand{\kg }{{\gamma}}

\newcommand{\kd }{{\delta}}

\newcommand{\kl }{{\lambda}}

\newcommand{\X}{\mathfrak{X}}
\newcommand{\Y}{\mathfrak{Y}}

\newcommand{\Qb}{\overline{Q}}
\newcommand{\ab}{\overline{a}}
\newcommand{\cR}{\mathcal{R}}
\newcommand{\bC}{{\bf C}}
\newcommand{\M}{\mathfrak{M}}
\newcommand{\SF}{\mathrm{SF}}

\begin{document}
\title{Wall-crossings for Twisted Quiver Bundles}

\author{Bumsig Kim}
\address{School of Mathematics, Korea Institute for Advanced Study, 85
Hoegiro, Dongdaemun-gu, Seoul, 130-722, Korea}
\email{bumsig@kias.re.kr}

\author{Hwayoung Lee}
\address{School of Mathematics, Korea Institute for Advanced Study, 85
Hoegiro, Dongdaemun-gu, Seoul, 130-722, Korea}
\email{hlee014@kias.re.kr}

\date{October 27, 2012}

\maketitle

\begin{abstract} Given a double quiver,
we study homological algebra of twisted quiver sheaves with the moment map relation using
the short exact sequence of Crawley-Boevey, Holland, Gothen, and King.
Then in a certain one-parameter space of the stability conditions, we obtain a
wall-crossing formula for the generalized Donaldson-Thomas
invariants of the abelian category of framed twisted quiver sheaves
on a smooth projective curve. To do so,
we closely follow the approach of Chuang, Diaconescu, and Pan
in the ADHM quiver case, which makes use of the theory of Joyce and Song.
The invariants virtually count framed twisted quiver sheaves with the moment map relation
and directly generalize the ADHM invariants of Diaconescu.
\end{abstract}

\section{Introduction}

A Nakajima's quiver variety is a holomorphic symplectic quotient
attached to a double quiver $\overline{Q}$, i.e.,
a quiver whose arrows are paired $(a,\bar{a})$ such that $\bar{a}$ is a
reverse arrow of $a$. This holomorphic symplectic quotient is a GIT quotient of
a locus defined by a moment map relation.
In \cite{K}, the moduli of stable twisted  quasimaps to the symplectic quotient from a fixed smooth projective curve $X$
is obtained as an application of the quasimap construction of \cite{CK, CKM} and shown to come
with a natural symmetric obstruction theory.
This result generalizes Diaconescu's work \cite{Dia}.

The stability  of stable twisted quasimaps turns out to be an asymptotic one in the one dimensional stability
parameter space $\RR _{> 0}$ of the abelian category $\cA '$ of framed twisted quiver sheaves on $X$.
It is therefore natural to investigate the wall-crossing phenomena of the moduli stack
$\M ^{ss}_\tau (\gamma)$ of $\tau$-semistable objects with numerical class
$\gamma$ in $\cA'$ as $\tau$ varies in $\RR _{>0}$.

For that study of wall-crossings, there are two theories available: the theory of Joyce and Song \cite{JS};
and the theory of Kontsevich and Soibelman \cite{KS}. In this paper, we perform our research
according to the framework of Joyce and Song.

First, we study the homological algebra of the category of twisted quiver sheaves {\em with the moment map relation} ~\eqref{M}.
This homological study is a generalization of works  Crawley-Boevey and Holland \cite{CH}, Crawley-Boevey \cite{CB1, CB2}, Gothen and King \cite{GK}, 
and Diaconescu \cite{Dia}.
We deduce that a truncated part of the category $\cA '$ behaves like a 3-Calabi-Yau category. For example, a
suitably defined antisymmetric bilinear form is numerical (see Proposition ~\ref{num}).
This property is the first main result of this paper and originates from the moment map relation.

Next, using the above bilinear form on the numerical $K$-group of $\cA '$, we define a Lie algebra $L(\cA ')$ and,
using the straightforward generalization of the Chern-Simons functional in
\cite{Dia}, we construct a Lie algebra homomorphism to $L(\cA ')$ from a Ringel-Hall type algebra of stack functions with algebra stabilizers
supported on virtual
indecomposables, as in \cite[Theorems 3.16 and 7.13]{JS}.
This Lie algebra homomorphism yields the definition of the generalized Donaldson-Thomas invariant for $(\overline{Q}, X,\gamma, \tau)$ via
the log stack function for $\M ^{ss}_\tau (\gamma)$.

Finally, using the approach of \cite{Dia, CDP1, CDP2}, we establish a wall-crossing formula
of the invariants (see Theorem ~\ref{wall}), which is the second main result of this paper.
In section ~\ref{Ja}, we show that the invariants vanish when the framing is zero and, at the same time,
the curve $X$ is not rational.
The wall-crossing correction  could be therefore nontrivial only when $X$ is rational.

\section{Homological Algebra}

The aim of this section is to prove a suitably defined antisymmetric Euler-like bilinear form
of framed twisted quiver sheaves is numerically determined in certain cases (see Proposition ~\ref{num}).
For this, we begin with finding a partial injective resolution ~\eqref{Vseq}
of double quiver representations with the moment map relation ~\eqref{MR1}.

\subsection{Double quivers} We set up notations for quivers.
Let $Q$ be a finite quiver, i.e., a directed graph whose arrow set $Q_0$ and vertex set $Q_1$
are finite. The tail map and the head map from $Q_1$ to $Q_0$ are denoted by $t$ and $h$, respectively. 
For each arrow $a\in Q_1$, we denote by $\bar{a}$  the reverse arrow of $a$.
Define the double quiver $\overline{Q}$ of $Q$ by adjoining a reverse arrow to each arrow of $Q$.
A path $p$ is an ordered set $a_1...a_m$ of arrows $a_i$ such that $ta_{i}=ha_{i+1}$ for $i=1,...,m-1$.
Define the head and the tail of $p$ by $hp=ha_1$ and $tp=ta_m$, respectively.
For each vertex $i\in Q_0$, define a trivial path $e_i$. Also set $h(e_i)=t(e_i)=i$.
The lengths of $p$ and $e_i$ are by definition $m$ and $0$, respectively.
Let $R$ be a commutative ring with unity $1$. The path algebra $R\Qb$ is the $R$-algebra
generated by all paths subject to relations by the following rules (see for instance \cite{GK}):
$p\cdot q = pq$ if $tp=hq$, $0$ otherwise; $p\cdot e_{tp}=p$; and $e_{hp}p=p$.

\subsection{Quiver representations}
In this section, we let $Q_0'$ be any nonempty subset of $Q_0$. We fix $\lambda _i
\in R$ for each $i\in Q_0'$.

Following  \cite{CH}, we consider the
two-sided ideal $(\mu-\lambda)$ generated by the relation
\begin{equation}\sum _{ i\in Q_0'} \left(\label{MR1}\left(\sum_{a\in \overline{Q}_1: ha=i}(-1)^{|a|}a \ab
\right) -\lambda_ie_i\right)=0, \end{equation}
 where $|a|$ is $0$ if $a\in Q_1$ or $1$ otherwise.
In fact, \eqref{MR1} is the functorial expression of the moment map equation (see \cite{CH, CB2})
and hence we call \eqref{MR1} {\em the moment map relation for the quiver} (or more precisely for $(\overline{Q}, Q_0', \lambda)$).
We will study the homological algebra of  the quotient algebra $A:=
R\Qb/(\mu -\lambda)$, i.e., the deformed preprojective algebra which was introduced by  W. Crawley-Boevey and M. Holland in \cite[Section 4] {CH}.
Note that every $A$-module
can be considered as an $R$-module by the natural $R$-module
homomorphism $R\ra A$, $r\mapsto r\sum _{i\in Q_0} e_i$.

Let $\cR$ be the
abelian category of (left) $A$-modules.
For $V\in \cR$, we construct  a sequence in $\cR$, which becomes a partial injective resolution of $V$ when $R$ is a field.
The sequence is defined to be
\begin{eqnarray}\label{Vseq}
0 &\ra&  V \stackrel{\epsilon}{\ra} \bigoplus _{i\in Q_0} \Hom _{R}(e_i A, V_i )  \\
&\stackrel{g}{\ra}& \bigoplus _{a\in \Qb _1}\Hom _{R} (e_{ta}A, V_{ha})
\stackrel{m}{\ra}  \bigoplus _{i\in Q_0'} \Hom _{R} ( e_i A, V_i )  ,\nonumber
\end{eqnarray}
where:
\begin{itemize}
\item we view $e_iA$ as an $R$-$A$ bimodule so that $\Hom _R (e_iA, V_j)$ is a {\em left} $A$-module by 
$a\alpha (p):=\alpha (pa)$ for $\alpha\in \Hom _R (e_iA, V_j), a\in A, p\in e_iA$;
\item $V_i := e_i V$ which is a $R$-submodule of $V$;
\item the $A$-homomorphisms $\ke$, $g$, $m$ are defined by: for $p_i\in e_iA\subset A$
\begin{enumerate}
\item $\epsilon (v)(p_i) = p_iv$;
\item $g(\ka  )_a ( p_{ta}) = \ka _{ha} (a p_{ta}) - a (\ka _{ta} (p_{ta}))$;
\item $m(\gamma ) ( p_i) = \sum _{a\in \Qb : ta = i} (-1)^{|a|} (\ab (\kg _a(p_i)) + \kg _{\ab} ( a p_i))$.
\end{enumerate}
\end{itemize}

\medskip

The above sequence ~\eqref{Vseq} is obtained from exact sequences in \cite[Lemma 4.2]{CH}, \cite  [Proof of Lemma 1]{CB1}, \cite  [Proof of Lemma 3.2]{CB2}, \cite[(2.1)]{GK}.
 Since $A$ has the `moment map relation' \eqref{MR1}, $g$ is not surjective (see \cite[(2.1)]{GK}). 
 Our moment map relation is slightly different from that in \cite{CH, CB1, CB2} 
 in that we only consider the relations in $Q'_0$, not the relations in $Q_0$.

\begin{Prop}\label{PropS}
The sequence ~\eqref{Vseq} is exact.
\end{Prop}
\begin{proof}
Following the proofs  in \cite{CH, CB1, CB2, GK}, we record a proof for the sake of completion.
If  $\epsilon(v)(e_i)=e_iv=0$ for all $i\in Q_0$, then $\sum _i e_i v=0$.
Since $\sum e_i=1$, we see that $v=0.$ Thus $\epsilon$ is injective. It is clear that $\mathrm{Im}\,\epsilon \subset \Ker\, g$.
Next we will show that $\Ker \, g\subset \mathrm{Im} \,\epsilon$. Consider $\ka \in \bigoplus _{i\in Q_0} \Hom _{R}(e_i A, V_i )
\subset \Hom _R(A, V)$ such that $g(\ka )=0$.  This implies that $\ka$ is $A$-linear. Therefore $\Ker\, g \subset  \Hom_A(A,V)=\epsilon (V)$.
Next we can check that $ \mathrm{Im}\, g\subset \Ker \, m$ since
\begin{eqnarray*}
&(m\circ g(\ka_i))_i(p_i)= \sum_{a\in \Qb : ta = i}(-1)^{|\ab|}  (\ab (g(\ka_i)_{a}(p_{i}))+g(\ka_i)_{\ab}(a p_{i}))=\\
&\sum_{a\in \Qb : ta = i}(-1)^{|\ab|}  (\ab (\ka_{ha}(a p_{i})-a\ka_{i}(p_i)) +\ka_{i}(\ab a p_{i})-\ab (\ka_{t\ab}( a  p_i)))=0.
\end{eqnarray*}
The last equality above follows  from the $R$-linearity of $\ka$, the moment map relation, and $t\ab=ha.$
Finally, let us show the hard part $\Ker\, m\subset \mathrm{Im}\, g.$ 
 Let $(\gamma_a)_{a\in \Qb}\in \Ker\, m$, in other words, $ \text{for each } i \in Q_0',$
 $\sum _{a\in \Qb : ta = i}(-1)^{|a|}(\ab(\gamma_a(p_{ta}))+\gamma_{\bar{a}}(a p_{ta}))=0.$
 Let $A_k$ be an $R$-module generated by arrows $p$ whose lengths are less than or equal to $k$.
 For example, $A_0$ is generated by $e_i$, $i\in Q_0$ and $A_1$ is generated by $a\in  \overline{Q}_1$ and $e_i$, $i\in Q_0$.
Then consider the filtration  $A_0\subset A_1\subset A_2\subset \cdots .$
  For each vertex $i$, there is the corresponding filtration
$e_iA_0\subset e_iA_1\subset e_iA_2\subset \cdots$.
An $R$-linear map $\ka$ will be constructed by induction on the filtration.
 Define $\ka=(\ka_i)\in \bigoplus _{i\in Q_0} \Hom _{R}(e_i A, V_i )  , \;\ka_i:e_iA \rightarrow V_i$ as
 \begin{itemize}
 \item $\ka_i \mid_{e_iA_0}=0$ and
 \item $\ka_i(ap_{ta})=a\ka_{ta}(p_{ta})+\gamma_{a}(p_{ta})$ for $a\in \Qb _1$ with $ha=i$.
 \end{itemize}
 We need to show that $\ka$ is well-defined.  For that, we consider the path algebra $F:=R\Qb$ without any relation
 and its corresponding filtration $F_i$. Note that $\ka$ is well defined as an element on $\Hom (F, V)$ and
 $\tilde{g}(\ka ) = \kg$ if
 $\tilde{g} : \bigoplus _{i\in Q_0}\Hom _{R}(e_i F, V_i )  \ra \bigoplus _{a\in \Qb _1}\Hom _{R} (e_{ta}F, V_{ha})$ is
 the homomorphism corresponding to $g$.
Now we claim that $\ka$ is well-defined on $A_n$.
It is clear that the claim is true when $n=0, 1$.
Let $F_n$ be the $R$-submodule of $F$ spanned by length-$n$ elements.
Suppose that $\ka$ is well-defined on $A_{n-1}$.
Then note that for $p_i\in e_iF_{n-2}$ and $i\in Q_0'$,
\begin{eqnarray} & & \ka_i ((\sum_{a\in \Qb _1, ha=i}(-1)^{|a|}a \ab-\lambda_i )p_i)\label{well-defined} \\
&=& \sum (-1)^{|a|}(a (\ka_{ta}(\ab p_i))+\gamma_{a}(\ab p_i))-\lambda_i \ka_i(p_i) =0 \nonumber \end{eqnarray}
since $\sum (-1)^{|a|} \gamma_{a}(\ab p_i) $ becomes
\begin{eqnarray*}
 & & 
  - \sum (-1)^{|a|}a\gamma_{\ab}( p_i)  \;\; (\textrm{by  }m(\gamma_a)=0) \\
&=& - \sum (-1)^{|a|}a (\tilde{g}(\ka_i)_{\ab})( p_i) \;\; (\textrm{by } \gamma =\tilde{g}(\ka ))  \\
&=& - \sum (-1)^{|a|} (a(\ka_{h\ab}(\ab p_i))-a (\ab \ka_{t\ab}(p_i)) )\;\;(\textrm{by the definition of } \tilde{g}).
\end{eqnarray*}
Combined with the inductive definition of $\ka$, the equation ~\eqref{well-defined}  implies that $\ka =0 $
 on the two-sided ideal $(\mu -\kl )$ of $A_n$. By the definition of $\ka$, it is clear that $g(\ka )= \gamma$. \end{proof}

Replacing $V$ by $W$ in  the sequence ~\eqref{Vseq} and taking $\Hom _A(V, \cdot )$ on \eqref{Vseq},
we obtain a sequence
\begin{eqnarray}\label{VWseq}
0 &\ra& \Hom _A (V, W) \ra \bigoplus _{i\in Q_0} \Hom _R (V_i,W_i) \\
&\ra& \bigoplus_{a\in \Qb _1} \Hom _R(V_{ta}, W_{ha} )
\ra \bigoplus_{i\in Q_0'} \Hom_R (V_i, W_i ) \nonumber
\end{eqnarray} of $R$-modules. This simplification follows from  the adjunction
of \cite[(2,2)]{GK}.
Let ${\bf C}(V, W)$ be the complex consisting of the last three terms of the sequence ~\eqref{VWseq}
with the first term at degree $0$. Then we get the following.

\begin{Cor} Let $R$ be a field.

\begin{enumerate}

\item $\mathrm{Hom}_R (e_iA, V_j)$ is an injective $A$-module.

\item For $l=0, 1$,
\[ \Ext ^l_{A} (V, W) \cong H^l ({\bf C} (V, W)). \]

\end{enumerate}

\end{Cor}

\begin{proof} By the above adjunction, we note the equivalence of functors:
$$\mathrm{Hom}_A(\bullet , \mathrm{Hom}_R (e_iA, V_j))\cong \mathrm{Hom}_R (e_iA\otimes _A\bullet , V_j ).$$
 The latter is   an exact functor  since $e_iA$ is a projective right $A$-module and $V_j$ is an injective $R$-module.
This proves (1).
Now (2) follows since \eqref{Vseq} is a partial injective resolution of $V$.
\end{proof}

\subsection{Quiver sheaves}
We carry out a similar procedure for twisted quiver sheaves in place of  quiver representations. To introduce
twisted quiver sheaves,
let $X$ be a Gorenstein projective variety,  let $\omega _X$ be
the dualizing sheaf for $X$,
and  let $\kl _i \in \Gamma (X, \omega _X)$.
 Suppose  also that we choose an invertible sheaf $M_a$ on $X$ for each $a\in \overline{Q}_1$ and
an isomorphism $f_{b,\bar{b}}: M_b \ot M_{\bar{b}} \ra \omega _X^\vee $ for
each $b\in Q_1$. We will set $f_{\bar{b} , b} = f _{b, \bar{b} }$ for $b\in Q_1$ using the natural isomorphism
$M_{\bar{b}} \ot M_b \cong M_b \ot M_{\bar{b}}$.
The condition  $M_b \ot M_{\bar{b}} \simeq \omega _X^\vee,\;  \forall b\in Q_1$ 
is a direct generalization of the corresponding condition in \cite{Dia}.

Provided with the above data, we define
an $\cO_X$-algebra structure  on the sheaf $$\bigoplus _{\text{all paths }p} M_p$$  by making:
\begin{itemize}
\item $M_p:=M_{a_1} \ot ...\ot M_{a_n}$ if $p=a_1...a_n$ with $a_i\in \Qb _1$, and  $M_{e_i}:=\cO_X$;
\item for $x_p \in M_p$, $x_q\in M_q$, let $x_px_q := x_p\ot x_q$ if $tp=hq$, and $0$ otherwise;
and in $\bigoplus _{p} M_p$ we have natural identifications $M_pM_{e_{hp}}=M_p\ot M_{e_{hp}}=M_p=M_{e_{tp}}M_p=M_{e_{tp}}\ot M_p$.
\end{itemize}
We denote by ${\bf M}\Qb$ this  $\cO_X$-algebra graded by lengths.

We want to define an ideal sheaf from the moment map relation.
First, for every local section $\xi \in \omega _X^\vee$, let
\begin{equation*}\label{M} (\mu -\kl )(\xi ):=
\sum _{i\in Q_0'} \left(\left( \sum _{a\in \Qb _1: ha = i} (-1)^{|a|} \xi _a \ot \xi _{\ab}\right)
 - \langle \xi , \lambda _i\rangle e_i \right),\end{equation*}
where $e_i$ stands for the
 constant $1$ in $M_{e_i}=\cO_X$ and
$\xi _a\ot \xi _{\ab}\in M_a\ot M_{\ab}$ is required to satisfy $f_{a,\ab}(\xi _a \ot \xi _{\ab})=\xi $. Then since
$(\mu-\kl)(f\cdot \xi ) = f\cdot (\mu -\kl )(\xi)$  for $f\in \cO_X$,
we can define the ideal sheaf  $(\mu -\kl)$ of  ${\bf M}\Qb$ generated by $(\mu -\kl)(\xi)$ for all $\xi \in \omega _X^\vee$
and hence the quotient sheaf $B:={\bf M}\Qb/(\mu - \kl)$.

\medskip

Now we turn into a homological algebra of the abelian category $\cA$ of $B$-modules. 
For an alternative and concrete description of a $B$-module, we view a $B$-module as an $\cO_X$-module as follows.
A collection of $\cO _X$-sheaves $E_i, i\in Q_0$ and
$\cO _X$-homomorphisms $\phi _a: M_a\ot E_{ta} \ra E_{ha}$, $a\in \overline{Q}_1$.
The collection
will be called
a {\em $\bf M$-twisted quiver sheaf} on $X$ in our context if the following moment map relation holds:
\begin{equation}\label{M}
\sum _{i\in Q_0'} \left(\left( \sum _{a\in \Qb _1: ha = i} (-1)^{|a|} \phi _a \circ (\mathrm{Id}_{M_{a}} \ot \phi _{\ab})\right)
 - \lambda _i\ot \mathrm{Id}_{E_i} \right)=0.\end{equation}

\begin{Prop}\label{McA}
The category of $\bf M$-twisted quiver sheaves is equivalent to $\cA$.
\end{Prop}
\begin{proof}
If $\{E_i, \phi _a\}$ is an $\bf M$-twisted quiver sheaf, then it is obvious how to give a $B$-module structure
on $\oplus E_i$. Conversely, for a $B$-module $E$, define $E_i:= e_iE = M_{e_i}E$ and $\phi _a: M_a\ot e_{ta}E \ra e_{ha}E$. 
Then it is simple to check the collection $\{ E_i, \phi _a\}$ has an induced $\bf M$-twisted quiver sheaf structure.
\end{proof}

We remark that a similar path algebra and a similar result were already  introduced and used 
in the original work on quiver representation theory over a ground field 
by P. Gabriel \cite{PG}; and the $\cO_X$-algebra ${\bf M}Q$ 
was introduced in \cite[Section 5]{AG2}
and also used in \cite[Section 3]{GK}; and 
Proposition \ref{McA} is a generalization of results of \cite{CH} (see Lemma 2.1 and Section 4)
in the context of preprojective algebras and Proposition 5.1 of \cite{AG2} 
 in the context of quiver sheaves with no moment map relations.

\medskip

For $E\in \cA$, there is an exact sequence in $\cA$
\begin{eqnarray}\label{Eseq}
& &  0             \ra   E \stackrel{\epsilon}{\ra} \bigoplus _{i\in Q_0} \cH om _{\cO _X}(e_i B, E_i )   \\
&\stackrel{g}{\ra}& \bigoplus _{a\in \Qb _1}\cH om _{\cO _X} ( M_a\otimes_{\cO _X}e_{ta}B, E_{ha})  \nonumber \\
&\stackrel{m}{\ra}&  \bigoplus _{i\in Q_0'} \cH om _{\cO _X} (\omega _X^\vee \otimes_{\cO _X}e_i B, E_i ),\nonumber
\end{eqnarray}
where:
\begin{itemize}
\item $e_i$ denotes the constant $1$ in $M_{e_i}=\cO_X$;
\item we view $e_iB$ as an $\cO_X$-$B$ bimodule so that $\Hom _R (e_iB, E_j)$ is a {\em left} $B$-module;
\item $E_i := e_i E$ which is a $\cO _X$-submodule of $E$;
\item the $B$-homomorphisms $\ke$, $g$, $m$ are defined by:
for $p_i\in e_iB\subset B$, $\xi = f_{\ab , a}(\xi _{\ab}\ot \xi _a)$,
\begin{enumerate}
\item $\epsilon (e)(p_i) = p_ie$;
\item $g(\ka  )_a (x_a\otimes p_{ta}) = \ka _{ha} (x_ap_{ta}) -\phi _a (x_a\ot \ka _{ta} (p_{ta}))$ for $x_a\in M_a$;
\item $m(\gamma ) (\xi \otimes p_i) =
\sum _{a\in \Qb : ta = i} (-1)^{|a|} ( \phi _{\ab}(\xi _{\ab} \ot \kg _a (\xi _a\otimes p_i))
+ \kg _{\ab} (\xi _{\ab}\otimes (\xi _a  p_i)) ) $.
\end{enumerate}
\end{itemize}

\begin{Prop}
The sequence {\em (\ref{Eseq})} is exact.
\end{Prop}

\begin{proof} The proof is parallel to the proof of Proposition ~\ref{PropS}.
\end{proof}
As before, we replace $E$ by $F$ in  the sequence (\ref{Vseq}) and take $\cH om _{B}(E, \cdot )$
to obtain a sequence
\begin{eqnarray}\label{EFseq}
&0&  \ra \cH om _B (E, F) \ra \bigoplus _{i\in Q_0} \cH om_{\cO_X} (E_i, F_i) \\
&\ra& \bigoplus _{a\in \Qb _1} \cH om_{\cO_X}(M_a\ot E_{ta}, F_{ha} )
\ra \bigoplus _{i\in Q_0'} \cH om _{\cO_X} (\omega _X^\vee\ot E_i, F_i ) \nonumber
\end{eqnarray}
of  $\cO_X$-modules.
Complexes  (\ref{VWseq}) and  (\ref{EFseq}) are  direct generalizations  of \cite[(3.2)]{Dia}.

Let ${\bf C}(E, F)$ be a complex consisting of the last three terms of \eqref{EFseq}. Assume that
$E$ is locally free, i.e., by definition, $E_i$ are locally free for all $i\in Q_0$,
In this case, we will observe that ${\bf C}(E, F)$
is quasi-isomorphic to a complex computing  $\Ext ^i_{B} (E, F)$ for $i=0, 1$.
In what follows, we denote the $i$-th hypercohomology group of the complex ${\bf C}(E, F)$ by ${\bf H} ^i (X, {\bf C}(E, F))$.

\begin{Cor}\label{12}  Assume that $E$ is locally free. Then, for $i=0, 1$,
\[ \Ext ^i_{B} (E, F) \cong {\bf H} ^i (X, {\bf C}(E, F)). \]
\end{Cor}

\begin{proof} Note that a partial injective resolution  $J^\bullet$ of $F$ can be
obtained from injective resolutions $I_i^\bullet$ of $F_i$ and \eqref{Eseq} since $\oplus _iI_i^k$ has
an induced $B$-module structure (see \cite[Section 3]{GK} for detail).
Note that $\cH om_B(E, J^{\bullet}) $ as an $\cO _X$-complex is
quasi-isomorphic (at $0, 1$) to \[ 0\ra \bigoplus _{i\in Q_0} E_i^\vee \ot F_i
\ra \bigoplus _{a\in \overline{Q}_1} M_a^\vee\ot E_{ta}^\vee \ot F_{ha} \ra
\bigoplus _{i\in Q_0'} \omega _X\ot E_i^\vee \ot F_i \] which is ${\bf C}(E,F)$. 
\end{proof}

We remark that Corollary 2.2 (2) and Corollary 2.5 directly  generalize Corollary 3.11 of  \cite{Dia}. 

\medskip

Now we introduce a bilinear form  $\langle , \rangle$ on  $\cA$ following \cite{JS}.
\begin{Def}
Define $\langle E, F\rangle$ to be
 \[ \dim \Ext ^0 _B(E, F) - \dim \Ext ^1 _B(E, F) + \dim \Ext ^1_B(F, E) - \dim \Ext ^0 _B(F, E).\]
\end{Def}

If $E_0$ denotes the $\cO _X$-coherent sheaf $\bigoplus _{i\in Q_0\setminus Q_0'} E_i$
(when $Q_0=Q_0'$, always $E_0=0$), we finally come to the first main result of this paper.

\begin{Prop}\label{num}\label{P} Suppose that $X$ is a smooth projective curve and $E$ and $F$ are locally free.
Assume either $E_0 =0$ or $F_0 =0$.
\begin{enumerate}
\item For $i=0, 1, 2, 3$., $${\bf H} ^i (X, \bC (E, F)) \cong {\bf H} ^{3-i}(X, {\bf C}(F, E))^\vee .$$
\item $\langle ,\rangle $ is numerically determined as follows:
\begin{eqnarray*}
&& \langle E,F\rangle \\  & = &\sum_{a\in{\overline{Q}_1}}(d(E_{ta})r(F_{ha})-d(F_{ha})r(E_{ta}) + d(M_a)r(E_{ta})r(F_{ha})\\
&&-(1-g) r(E_{ta})r(F_{ha}))+ 2\sum_{i\in Q_0'} (-d(E_i)r(F_i)+d(F_i)r(E_i)),
\end{eqnarray*}
where $d(E_i)$ and $r(E_i)$ stand for the degree and the rank of the locally free sheaf $E_i$, respectively.

\end{enumerate}
\end{Prop}

\begin{proof}
(1): Note that $\bC (F, E) ^\vee \ot \omega _X= \bC (E, F)[2]$. Combined with the Serre duality, this  implies
${\bf H} ^i (X, {\bf C}(E, F)) \cong {\bf H} ^i (X, \bC (F, E)^\vee \ot \omega _X [-2])
\cong {\bf H} ^{3-i} (X, \bC (F, E) ) ^\vee $.

(2):  Note that  $\chi (X, \bC (E,F)) = \langle E, F\rangle $ by (1) above. On the other hand,
$\chi (X, \bC (E,F))$ is equal to $\chi (X, \bC ^0(E,F)) - \chi (X, \bC ^1(E,F))+ \chi (X, \bC ^2(E,F))$, hence
the topological expression for $\langle E,F\rangle$ follows from the Riemann-Roch formula.
\end{proof}


\section{Wall-crossings}

From now on, let $X$ be a smooth projective curve, let $\lambda _i=0$ for all $i\in Q_0'$, and let
$Q_0\setminus Q_0' = \{ 0\}$. In the abelian category $\cA '$ of twisted quiver sheaves, we will
consider stability conditions for $\tau\in \RR _{>0}$ and, in the framework of Joyce-Song theory \cite{JS}, we will
define the generalized Donaldson-Thomas invariants
using the moduli space of $\tau$-semistable objects in $\cA '$. We will derive a wall-crossing formula
following the approach of  Chuang, Diaconescu, and Pan \cite{Dia, CDP1, CDP2}.

\subsection{Chamber structures}\label{chamber} In this subsection, for $\tau\in \RR _{>0}$
we introduce the notion of a $\tau$-stability on twisted quiver sheaves
and show that for each fixed numerical class with a minimal framing
the stability space $\RR _{>0}$ has a finite number of critical values.
The precise definition of critical values is not important. The relevant required property will be only that
there are no strictly $\tau$-semistable quiver sheaves for every noncritical value $\tau$.

Let $K$ be a nonzero complex vector space.
Denote by $\cA '$ the abelian category of $\bf M$-twisted quiver sheaves $E$ with
$E_0=K^S\ot\mathcal{O}_X$ for some finite set $S$ (depending on $E$).
In this category $\cA'$, a morphism from $(E_i, \phi _a)$
to $(E_i', \phi _a')$ is
by definition a usual morphism as $\bf M$-twisted quiver sheaves
with the framing condition that the attached $\mathcal{O}_X$-homomorphism $K^S\ot \mathcal{O}_X\ra K^{S'}\ot \mathcal{O}_X$
is a block matrix $(c^{s,s'})_{(s,s')\in S\times S'}$, $c^{s,s'}\in \CC$.
It is straightforward to check that
the category $\cA '$ is an abelian category.

Let $E\in \cA '$ and
$\tau \in \RR _{>0}$ be the stability parameter.
For a  nonzero $\bf M$-twisted quiver sheaf $E$ we define the $\tau$-slope of $E$ to be
$$\mu_{\tau}(E):=\frac{\deg(\bigoplus _{i\ne 0} E _i)}{\rank(\bigoplus _{i\ne 0} E _i)}
+\frac{\tau \cdot \rank E_0}{\rank (\bigoplus _{i\ne 0} E _i))}\in (-\infty, \infty ].$$

\begin{Def}
A nonzero object $E$ of $\cA '$ is called {\em $\tau$-(semi-)stable}
if $\mu_\tau(F)(\leq)< \mu_\tau(E)$ for any nonzero proper subobject $F$ of $E$.
\end{Def}
The definitions of $\tau$-slope of $E$ and $\tau$-(semi-)stability in Definition 3.1 
are generalizations of a slope function and a (semi-)stability condition introduced in \cite{Dia}, 
and furthermore, belong to the class of slope stability conditions for quiver sheaves introduced 
in \cite{AG1,{AG2}} for a larger family of parameters.

\medskip

Let
$r=\rank (\bigoplus _{i\ne 0} E _i)$, $v=\rank E_0$, and $d=\deg (\bigoplus _{i\ne 0} E _i)$.
If $E$ is strictly $\tau$-semistable, that is, $\tau$-semistable but not $\tau$-stable, then $\tau$ must be of form
\begin{eqnarray}\label{criticaltau}
&\tau=\frac{rd'-r'd}{r'v}\; {\rm or}\; \tau=\frac{r'd-rd'}{(r-r')v} \label{Cr}
\end{eqnarray}
for some $r', d'\in \ZZ$ with $1\leq r' \leq r-1$.

Let us consider the set $C(\cA ') =\dim K \cdot \mathbb{N} \times (\NN \times \ZZ )^{Q_0'}$ of numerical classes of
twisted quiver bundles. 
The class of $E$ has $\rank E_0$ at the first entry, $\rank E_i$ at the middle one for $i\in Q_0'$, and
$\deg E_i$ at the last one for $i\in Q_0'$.
Let $\gamma \in C(\cA ')$ and $v_0(\gamma )$ be the first entry of $\gamma$. Suppose that $v_0(\gamma )=\dim K$. 
Then by a generalization of \cite[Lemma 4.7]{Dia} there is a number $N(\gamma )$ such that
there are no strictly $\tau$-semistable objects with numerical class $\gamma $ if $\tau\ge N(\gamma )$.
Let $C(\gamma )$ be the set of all possible positive values $\tau \le N(\gamma )$ in (\ref{criticaltau})
so that for $\tau\notin C(\gamma)$ there are no strictly $\tau$-semistable objects with $\gamma$-class.
Note that $C(\gamma )$ has no accumulation points in $\RR$. Hence, $C(\gamma )$ is a finite set.
We call an element of $C(\gamma )$ a {\em critical value}.

\subsection{Chern-Simons functionals}
Let $\M$ be the moduli stack parameterizing all objects $E$ of $\cA '$.
In order to apply the Joyce-Song theory, we need a local description
of $\M$ as a critical locus of a holomorphic function on a complex domain
(see \cite[Theorems 5.4 and 5.5]{JS} which makes use of Miyajima's results in \cite{M}).
The theorems below are straightforward generalizations of
\cite[Theorems 7.1 and 7.2]{Dia}. In particular, the Chern-Simons functional  (\ref{CS}) is a 
direct generalization of one in \cite[(7.7)] {Dia}.

Let $\cA ' _{\le 1}$ be a subcategory of $\cA '$ of an object $E$ with $(E)_0=K \otimes \cO_X$ or $0$. 
Let $\mathcal{M}^{si}$ be the coarse moduli space of  simple objects in $\cA ' _{\le 1}$.
\begin{Thm}\label{ld1}
For every $[E]\in  \mathcal{M}^{si}(\mathbb{C})$, the analytic germ of
$\mathcal{M}^{si}(\mathbb{C})$ at $[E]$ is isomorphic to $(\mathrm{Crit}(f), u)$ for some
holomorphic function
$f:U \rightarrow \mathbb{C}$ on a finite dimensional complex manifold $U$, where $u$ is a point of $U$.
\end{Thm}

Let $S$ be an $\mathrm{Aut} (E)$-invariant subscheme of $\mathrm{Ext}^1 _{\cA '}(E, E)$ parameterizing
a versal family of objects in $\M (\CC )$ near $E$.
\begin{Thm}\label{ld2}
For every $E\in  \M (\mathbb{C})$ and  a maximal compact
subgroup $G$ of $\mathrm{Aut}(E)$, the analytic germ of
$(S, 0)$ is $G^{\CC}$-equivariantly isomorphic to $(\mathrm{Crit}(f), 0)$ for some $G^{\CC}$-invariant
holomorphic function
$f:(\mathrm{Ext} ^1 _{\cA '}(E, E), 0) \rightarrow (\mathbb{C}, 0)$, where $G^{\CC}$ is the complexification of
$G$ in $\mathrm{Aut} (E)$.
\end{Thm}

The proofs of \cite[Theorems 7.1 and 7.2]{Dia} work for the general case after
the replacement of the Chern-Simons functional \cite[(7.7)]{Dia} according to
the double quiver $\overline{Q}$. In what follows, we describe the Chern-Simons functional for the general case.

Let $E=(E _i, \phi _a)_{i\in Q_0, a\in \overline{Q}_1}$ be a framed twisted quiver bundle on $X$ and let $\hat{X}$ denote
the complex manifold associated to $X$.
Then there is the
gauge-theoretical interpretation $(\hat{E}_i, \bar{\partial}_{E_i}, \phi _a^0)_{i\in Q_0', a\in \overline{Q}_1}$ of $E$, i.e.,
$\hat{E_i}$ is $E_i$ regarded as a $C^{\infty}$ complex vector bundle on $\hat{X}$,
$$ \bar{\partial}_{E_i} :C^{\infty}(\hat{E}_i) \rightarrow C^{\infty}(\hat{E}_i\otimes  \Lambda^{0,1} T^{*}_{\hat{X}})$$ 
is the unique semiconnection on $\hat{E_i}$ such that local holomorphic sections of $E_i$ are translated
into horizontal sections of $\bar{\partial}_{E_i}$,
and $\phi_a^0 \in C^{\infty}(\hat{M}_{ta}^\vee \ot \hat{E}_{ta}^\vee \ot \hat{E}_{ha})$ corresponds  to $\phi _a$.
Here the $(0,1)$-part of a usual connection is called a semiconnection (see \cite[Definition 9.1]{JS}).
Note that the flatness of $\bar{\partial}_{E_i}$ automatically holds  since $X$ is a curve.

 Now, the Chern-Simons functional $CS$ near $E$ is defined by
a generalization of \cite[(7.7)] {Dia}:
 \begin{eqnarray}\;\;\;\;\;\;\;\;\;\ \label{CS}
 CS(A_i,\varphi_a)=\int_X  \mathrm{Tr}(
\sum _{a\in Q_1} \varphi_{\ab}  \bar{\partial} _a\varphi_a
+\sum_{a\in \overline{Q}_1, ha\in Q_0'}(-1)^{|a|} A_{ha}\tilde{\phi}_{a}\tilde{\phi}_{\ab})
\end{eqnarray}
for
$$(A_i, \varphi _a) \in \prod_{i\in Q_0', a\in \overline{Q}_1} C^{\infty}(\mathrm{End}(\hat{E}_i)\otimes   \Lambda^{0,1} T^{*}_{\hat{X}})
\times C^{\infty}(\hat{M}_{ta}^\vee\ot \hat{E}_{ta}^\vee \ot \hat{E}_{ha}). $$
Here $\bar{\partial} _a$ is the semiconnection on $\hat{E}_{ta}^\vee\ot \hat{E}_{ha}$,
 $\tilde{\phi} _a = \phi _a ^0 + \varphi _a$, and the products in the integrand are naturally given by compositions and cup products
so that after all they are considered as elements in $\mathrm{End}(\hat{E}_{ta})\ot  \Lambda^{1,1} T^{*}_{\hat{X}}$.
The Chern-Simons functional is gauge-invariant and its critical equations are 
\begin{eqnarray}\label{ceq1} \bar{\partial}_a \varphi _a - \tilde{\phi}_a A_{ta} + A_{ha}\tilde{\phi}_{a} &=& 0, \;\forall a\in \overline{Q}_1;\\
   \label{ceq2}      \sum _{ha=i} (-1)^{|a|} \tilde{\phi }_a \tilde{\phi}_{\ab} &=&0, \; \forall i\in Q_0'.\end{eqnarray}
We note that \eqref{ceq1}  is the holomorphic condition on $\tilde{\phi} _a$ with respect to the new semiconnection
$(\bar{\partial} _{E_{ta}} + A_{ta})^\vee\ot (\bar{\partial} _{E_{ha}}+A_{ha}) $
and  \eqref{ceq2}  is the moment map relation on $\tilde{\phi} _a$.

\subsection{Ringel-Hall type algebras}

Recall that $\M$ denotes the moduli stack parameterizing all objects $E$ of $\cA '$.
We call a pair $(\X, \rho )$ a $\M$-valued stack function if
$\X$ is an Artin stack over $\CC$ and $\rho : \X\ra \M$ is a representable 1-morphism.
Let $\SF (\M)$ be the \lq Grothendieck group'  of $\M$-valued stack functions, i.e.,
the quotient  group of the free abelian group generated by stack functions, whose quotient is given
by the subgroup spanned by all elements of form 
\[    (\X, \rho ) - ((\Y, \rho_{ |_\Y}) + (\X \setminus \Y, \rho _{|_{\X \setminus \Y}})) \]
for a closed substack $\Y$ of $\X$.

There is a multiplication structure on $\SF (\M )$ for which the multiplication
$$  (\X _1, \rho _1) * (\X _2, \rho _2) $$ is defined to be the fiber product $(\X, \rho )$
in diagram
\[ \begin{CD}  \X  @>>> \mathfrak{Exact}(\M ) @>>{\pi _2}> \M \ ,\\
        @VVV @VV{\pi _1\times \pi _3}V @. \\
        \X _1 \times _{\CC}\X_2 @>>{(\rho _1,\rho _2 )}> \M \times _{\CC}\M@.
        \end{CD}\]
where:
\begin{itemize}
\item $ \mathfrak{Exact}(\M )$ is an Artin stack parameterizing short exact sequences in $\M$
and $\pi _i$ is the obvious $i$-th projection;
\item the square is the fiber product and $\rho$ is the composition of the upper arrows.
\end{itemize}
The multiplication is associative by \cite[Theorem 5.2]{J2}.
The induced algebra $\SF (\M )$ is called the {\em Ringel-Hall type algebra}.

Denote by $\M _{\ge 2}$ the moduli stack parameterizing all objects $E$ of $\cA '$ with $\rank E_0\ge 2\dim K$.
Let $\SF (\M _{\le 1}' )$ be the quotient algebra of $\SF (\M )$ factored by the ideal generated
by all $\rho : \X\ra \M$ which factor though $\M _{\ge 2}$. Finally we consider the subalgebra
$\SF (\M _{\le 1} )$ of $\SF (\M _{\le 1}' )$ generated by all $\rho : \X\ra \M$ which factor though
the moduli stack of locally free objects.

By Proposition ~\ref{num}, we may define an antisymmetric bilinear form
$\langle  , \rangle : (\dim K \cdot \NN \times (\NN \times \ZZ ) ^{Q_0'} )^2\ra \ZZ $.
Let $L(\cA ')$ be the $\QQ$-vector space with basis
$e_{\gamma}$, $\gamma \in \{0, \dim K \} \times (\NN \times \ZZ ) ^{Q_0'}$, equipped with
a Lie algebra structure given by $$[e_{\gamma}, e_{\tilde{\gamma}}]:=
\left\{\begin{array}{ll} (-1)^{ \langle \kg,\tilde{\kg}\rangle }\langle \kg, \tilde{\kg}\rangle
e_{\gamma+\tilde{\gamma}} & \text{ if } \gamma _0 + \tilde{\gamma} _0 \le \dim K ,\\
             0 & \text{ otherwise}\end{array}\right.$$
             (see \cite[Definition 5.13]{JS}).

In the below, we let $B\CC ^*$ denote the classifying stack of the multiplicative group $\CC ^*$.
By the local descriptions, Theorems ~\ref{ld1} and ~\ref{ld2}, of the moduli spaces we will have this.

 \begin{Thm}\label{LieHom}
There is a Lie algebra homomorphism
\[ \Psi : \SF _{\mathrm{alg}}^{\mathrm{ind}} (\M _{\le 1}) \ra L(\cA ') \]
satisfying  \[\Psi ([Z\times B\CC ^*,\rho ]) = -\chi(Z  , \rho ^*\nu ^B_{\M_{\le 1}} ) e _{\gamma} ,\]
where:

\begin{itemize}

\item $\SF _{\mathrm{alg}}^{\mathrm{ind}}(\M _{\le 1})$ is a certain subalgebra of $\SF (\M _{\le 1})$, spanned
by stack functions with algebra stabilizers supported on virtually indecomposable objects;
\item $Z$ is a variety and $\rho ^*\nu ^B_{\M_{\le 1}}$ is the $\ZZ$-valued constructible
function induced from the Behrend
function $\nu ^B_{\M _{\le 1}}$ for $\M _{\le 1}$;

\item $\chi (Z  , \rho ^*\nu ^B_{\M_{\le 1}}  )$ is defined to be the weighted topological Euler characteristic
$\sum _{n\in\ZZ } n \chi( (\rho ^*\nu ^B_{\M_{\le 1}})^{-1}(n))$.

\end{itemize}

\end{Thm}

\begin{proof} 
 Theorem 3.2 and Theorem 3.3  imply that
the Behrend function $\nu ^B_{\M _{\le 1}}$ on $\M _{\le 1}$ has the same property as one in \cite[Theorem 5.11]{JS}
and \cite[Theorem 7.4]{Dia}. 
In the analogy with \cite[Theorem 5.14]{JS}, we get the result (see also \cite[Section 2.2 and Theorem 3.2]{CDP1}).
 \end{proof}

\subsection{Harder-Narasimhan filtrations}

In this section, using Harder-Narasimhan filtrations
we express the stack function representing $\M ^{ss}_{\tau _-}$
in terms of those representing $\M ^{ss}_{\tau _+}$ and unframed moduli spaces.
By a purely algebraic Lemma in \cite{CDP1}
this expression induces a
wall-crossing formula. From now on, we let
$\tau _0\in C(r,d)$, $\tau_+ >\tau _0 $ and $\tau_-<\tau _0$ such that
there are no critical values between intervals $(\tau_0, \tau _+  ]$ and $[\tau_- , \tau _0  )$.

\medskip

For  $E \in \cA '$ and $\tau \in \RR _{>0}$, it is easy to see that
there is a unique filtration
$$0=E_0 \subset E_1 \subset E_2 \ldots \subset E_n=E$$
such that $E_k/E_{k-1}$ is $\tau$-semistable
and  $\mu_\tau(E_{k-1}/E_{k-2})>\mu_\tau(E_k/E_{k-1})$ for  $k=1,\ldots n$.
This so-called {\em Harder-Narasimhan filtration} will lead us to the following.

\begin{Lemma}\label{HNf} Let $E\in \cA$ with $(E)_0=K\ot\cO_X$. TFAE.

\begin{enumerate}

\item $E$ is $\tau _0$-semistable

\item $E$ is $\tau _+$-semistable or
there is a unique subobject $E'$ of $E$ satisfying:
$E'$, $E/E'$ are $\tau _+$-semistable,  $(E')_0 = K\ot \cO _X$,
$\mu _{\tau _+}(E') > \mu _{\tau _+} (E/ E')$, and  $\mu _{\tau _0}(E') = \mu _{\tau _0} (E/ E')$.

\item
$E$ is $\tau _-$-semistable or
there is a unique subobject $E'$ of $E$ satisfying:
$E'$, $E/E'$ are $\tau _-$-semistable, $(E/E')_0 = K\ot \cO _X$,
$\mu _{\tau _-}(E') > \mu _{\tau _-} (E/ E')$,  and $\mu _{\tau _0}(E') = \mu _{\tau _0} (E/ E')$.

\end{enumerate}

\end{Lemma}

\begin{proof} (1) $\Rightarrow$ (2).  Let E be $\tau_0$-semistable.
Let $0=E_0 \subset E_1 \subset ... \subset E_n=E$ be the $\tau _+$ Harder-Narasimhan filtration of $E$.
We take $E':=E_1$.
 If $n=2$, $(E_1)_0\ne 0$ for otherwise,
 $\mu_{\tau _0}(E_1) = \mu _{\tau _+}(E_1) > \mu _{\tau _+}(E_2/E_1) \ge \mu _{\tau _0}(E_2/E_1)$ which
 is a contradiction to the $\tau _0$-semistability of $E$. We prove that $n$ cannot be larger than 2.
Suppose that $n\ge 3$, then there are  $i, j$ such that $1\le i< j\le n$ and
\begin{equation}\label{ij} \mu _{\tau _+} (E_i/E_{i-1})
= \mu _{\tau _0} (E_i/E_{i-1}) >  \mu _{\tau _+} (E_j/E_{j-1}) = \mu  _{\tau _0} (E_j/E_{j-1}) .\end{equation}
This induces a contradiction as follows.

a) When $(E_1)_0=0$, then $\mu_{\tau _0}(E_1)=\mu _{\tau _+}(E_1) > \mu _{\tau _+}(E)\ge \mu _{\tau _0} (E)$.
This is a contradiction to the $\tau _0$-semistability
of $E$.

b) When $(E_1)_0=K\ot\cO _X$, then $(E_i)_0=K\ot\cO_X$ for all $i\ge 1$ so that the ineqaulity
$\mu _{\tau _+}(E_i)>\mu _{\tau _+}(E)$ implies that $\mu _{\tau _0}(E_i)=\mu _{\tau _0}(E)$ for all $i$. This contradicts \eqref{ij}.

Now, by the uniqueness of Harder-Narasimhan filtrations, the proof of (1) $\Rightarrow$ (2) follows.

(2) $\Leftarrow$ (1). The nontrivial case is that $E$ is not $\tau _+$-semistable. Let $F$ be a nontrivial subobject of $E$ and let $F':=\mathrm{Ker}(F\ra E/E')$.
Then $\mu _{\tau _+} (F')\le \mu _{\tau _+}(E')$ (because $E'$ is $\tau _+$-semistable) and $\mu _{\tau _+}(F/F') \le
\mu _{\tau _+}(E/E')$ (because $E/E'$ is $\tau _+$-semistable). Now take the limit $\tau _+\ra \tau _0$ to the both inequalities in order to
conclude that $\mu _{\tau _0}(F)\le \mu _{\tau _0}(E)$ since  $\mu _{\tau _0}(E') = \mu _{\tau _0} (E/ E')$.

(1) $\Leftrightarrow$ (3). This follows by an argument similar to the proof of (1) $\Leftrightarrow$ (2).
\end{proof}

Let $\delta_{\tau}(\gamma)$ denote the stack function $[ \M^{ss}_{\tau}(\gamma),\rho]\in \SF (\M _{\le 1})$
for the natural open embedding $\rho$ of the moduli stack  $\M^{ss}_{\tau}(\gamma) \subset \M$
of  $\tau$-semistable objects of $\cA '$ with the numerical class $\gamma \in C(\cA ')$.
We use notation $\delta (\gamma )$ for $\delta _{\tau}(\gamma )$ if $v_0(\gamma ) = 0$.
Then the previous lemma will induce relationships between
$\delta _{\tau _{\pm}}(\gamma )$, $\delta _{\tau _0}(\gamma )$, and $\delta (\gamma )$
in the Ringel-Hall type algebra as
in Lemma ~\ref{Fi} below. Before describing the lemma, we will need the following index sets.
For $l\ge 1$, let \begin{eqnarray*} HN_+(\gamma , \tau _0 , l) =  \{ (\gamma _1,...,\gamma _l ) | \ \gamma _i \in C(\cA ), \\
\sum _i \gamma _i = \gamma ,    v_0(\kg _1) = v_0(\kg ),
\mu _{\tau _0} (\gamma _i) =\mu _{\tau _0}(\gamma )\ \forall i    \} \nonumber\end{eqnarray*}
 and \begin{eqnarray*} HN_-(\gamma , \tau _0 , l) =  \{ (\gamma _1,...,\gamma _l ) | \ \gamma _i \in C(\cA ), \\
\sum _i \gamma _i = \gamma ,    v_0(\kg _l) = v_0(\kg ),
\mu _{\tau _0} (\gamma _i) =\mu _{\tau _0}(\gamma )\ \forall i    \}. \nonumber\end{eqnarray*}

\begin{Lemma} \label{Fi} In  $\SF (\M _{\le 1})$, the followings hold.
\begin{enumerate}

\item \begin{eqnarray*} \delta _{\tau_0}(\gamma) &=& \delta _{\tau_+}(\gamma ) + \sum _{(\kg _1,\kg _2) \in HN_+(\kg, \tau _0, 2)}
 \delta _{\tau_+}(\gamma _1) * \delta  (\gamma _2) .\\
  \delta _{\tau_0}(\kg ) &=& \delta _{\tau_-}(\gamma ) + \sum _{(\kg _1,\kg _2) \in HN_-(\kg, \tau _0, 2)}
 \delta (\gamma _1) * \delta _{\tau_-} (\gamma _2) .\end{eqnarray*}

 \item \begin{eqnarray*} \kd _{\tau _+}(\kg ) &=& \sum _{l\ge 1}  (-1)^{l-1} \sum _{HN_+(\gamma , \tau _0, l )}
 \kd _{\tau _0} (\kg _1) * \kd (\kg _2) * ... * \kd (\kg _l) .\\
  \kd _{\tau _-}(\kg ) &=& \sum _{l\ge 1} (-1)^{l-1} \sum _{HN_-(\gamma , \tau _0, l )} \kd  (\kg _1) * \kd (\kg _2) * ... * \kd _{\tau _0} (\kg _l).
 \end{eqnarray*}
 \item
 \begin{eqnarray*} & \kd _{\tau _-} (\kg ) = \kd _{\tau _+}(\kg )  + \\  &\sum _{l\ge 2} (-1)^{l-1} \sum
 _{HN_-(\gamma , \tau _0, l )} \kd (\kg _1)*...*\kd (\kg _{l-2}) * [\kd (\kg _{l-1}), \kd _{\tau _+} (\kg _l )  ] .\nonumber
 \end{eqnarray*}
\end{enumerate}
\end{Lemma}
\begin{proof} There are only finite nontrivial terms in each summation of (1), (2), and (3).
For example, when $l=2$, let us
consider an exact sequence $$0\ra E^1\ra E^2 \ra E^2/E^1 \ra 0$$ whose factors $E^1, E^2/E^1$ are $\tau _0$-semistable
with $\mu _{\tau_0}(E^1) = \mu _{\tau_0} (E^2/E^1)$. 
It suffices to show that $\deg (E^1)_j$ is bounded above by a number depending only on class $\gamma$.
Since $E^2$ is $\tau _0$-semistable, by a generalization of \cite[Lemma 2.4]{Dia}
$E^2$ is isomorphic to an element in a bounded family of vector bundles on $C$. 
Now using a finite covering $\pi: C\ra \PP ^1$, we see that
$\deg \pi _*(E^1)_j$ is bounded above, hence so is $\deg( E^1)_j$.

The statement (1) follows from Lemma ~\ref{HNf}.

For (2),  we rewrite the first equation of (1) as
\begin{equation}\label{delta+} \delta _{\tau_+}(\gamma) = \delta _{\tau _0}(\gamma ) - \sum _{(\kg _1,\kg _2) \in HN_+(\kg, \tau _0, 2)}
 \delta _{\tau _+}
(\gamma _1) * \delta  (\gamma _2)\end{equation}
  and  apply \eqref{delta+} to  $\delta _{\tau _+}(\gamma _1)$. This iterated procedure must stop since there are only finite nontrivial terms in
  the summation. This proves (2).

To prove (3), we start with the second equation of (2)
and replace $\delta _{\tau _0}(\gamma _l)$ by the first equation of (1).
\end{proof}

\subsection{Log stack functions} Following \cite[Definition 8.1]{J3},
we define the log stack function for $\gamma \in C(\cA ')$ as
\begin{eqnarray*}
\epsilon _{\tau} (\gamma ) := \sum_{\l\ge 1}
\frac{(-1)^{l-1}}{l} \sum _{\sum \kg _i = \kg, \mu _{\tau}(\kg _i) = \mu _{\tau }(\kg) \forall i}
\kd _{\tau} (\gamma _1) *\cdots * \kd _{\tau} (\gamma _l ).
\end{eqnarray*}
As in the proof of Lemma ~\ref{Fi}, there are only finite nontrivial terms in
the sum expression of $\epsilon _{\tau} (\gamma )$.
According to  \cite[Theorem 3.11]{JS},  the log stack function
$\epsilon _{\tau}(\kg )$ is an element in $\SF _{\mathrm{alg}}^{\mathrm{ind}}(\M _{\le 1})$.
In the below, if $v_0(\gamma )=0$, we let $\epsilon (\gamma)$ denote $\epsilon _{\tau }(\gamma)$.

\begin{Lemma}\label{log}
\begin{eqnarray*}
\epsilon _{\tau _-}(\kg ) -\epsilon _{\tau _+}(\kg)= \sum _{l\ge 2} \frac{(-1)^{l-1}}{(l-1)!} \sum _{HN_-(\gamma , \tau _0, l )}
[ \epsilon (\kg _1 ) , [ ... [\epsilon (\kg _{l-1}),\epsilon _{\tau _+} (\kg _l )]...] .
\end{eqnarray*}
\end{Lemma}

\begin{proof} Lemma \ref{Fi} (3) and the definition of log stack functions enable us to apply (the combinatorial argument of) \cite[Lemma 2.4]{CDP1}.
\end{proof}

Theorem ~\ref{LieHom} implies that for fixed  $\kg$,
 the $\tau$-invariant $J_{\tau}(\kg)\in \mathbb{Q}$  can be defined  as
$$ \Psi( \epsilon _{\tau}(\kg ) )=-J_{\tau}(\kg)e_{\kg}  .$$
We call $J_{\tau}(\kg)$ {\em the generalized DT-invariant} for $(\overline{Q}, X, \gamma ,\tau)$ (see \cite[Definition 5.15]{JS}).
This is a direct generalization of ADHM invariants by \cite[Lemma 3.1, Theorem 3.2] {CDP1}.
When the rank $v_0(\gamma)$ of the frame is 0, then the invariant will be denoted simply by $J(\kg)$ since it does not depend on $\tau.$
Now by combining Lemma ~\ref{log} and Theorem ~\ref{LieHom}, we conclude the following wall-crossing formula.
\begin{Thm}\label{wall}
\begin{eqnarray*}
 & (J_{\tau _-}(\kg) - J_{\tau _{+}}(\kg ) )e_{\kg}  \\  & =  \sum _{l\ge 2} \frac{1}{(l-1)!}
\sum _{HN_-(\gamma , \tau _0, l )} [ J(\kg _1)e_{\kg _1}, [ ... [J(\kg _{l-1})e_{\kg _{l-1}},J _{\tau _+} (\kg _l )e_{\kg _l}]...] .
\end{eqnarray*}
\end{Thm}

\subsection{The action by the Jacobian variety}\label{Ja}
Suppose that $v_0(\gamma )=0$ and the genus of $X$ is $g\ge 1$. In this case, by the similar argument for the proof of
\cite[Proposition 6.19]{JS}, the generalized DT invariant $J(\gamma)$ vanishes as follows.
Let $\mathcal{J}(X)$ be the Jacobian variety of $X$ and let $L$ be the universal line bundle on $X\times \mathcal{J}(X)$.
Then the torus group $\mathcal{J}(X)$ acts on $\M ^{ss}_{\tau }(\gamma )$
by $t\cdot E:= E\ot L_t$ for $t \in \mathcal{J}(X)$. Note that this action yields
a torus fibration on $\M ^{ss}_{\tau }(\gamma )$ and the Behrend function on $\M ^{ss}_{\tau }(\gamma )$ is constant on
each $\mathcal{J}(X)$-orbit. Hence we conclude that $J(\gamma )=0$ using the expression of $J(\gamma )$
by the weighted Euler characteristics (see \cite[Section 5.3]{JS}).

\medskip

\noindent{\bf Acknowledgments.} We would like to express our deep gratitude to Kurak Chung for
useful discussions and to Emanuel Diaconescu and Jae-Hyouk Lee for their invaluable comments.
We also thank a referee for careful comments which improve the presentation of the paper.
This work is financially supported by NRF-2007-0093859.



\begin{thebibliography}{99}




\bibitem{AG1} L. Alvarez-Consul and O. Garcia-Prada, {\em Dimensional reduction and quiver bundles,}  J. Reine Angew. Math. 556 (2003), 1-46.

\bibitem{AG2} L. Alvarez-Consul and O. Garcia-Prada, {\em Hitchin-Kobayashi correspondence, quivers, and vortices,}
 Comm. Math. Phys. 238 (2003), no. 1-2, 1-33.



\bibitem{CDP1}      W-E. Chuang, D.E. Diaconescu, and G. Pan,
         {\em Chamber structure and wallcrossing in the ADHM theory of curves II,}
        Journal of Geometry and Physics 62 (2012), 548-561.

\bibitem{CDP2}      W-E. Chuang, D.E. Diaconescu, and G. Pan,
                          {\em  Rank two ADHM invariants and wallcrossing,} Communications in Number Theory and Physics, 4 (2010), 417-461.


\bibitem{CK} I. Ciocan-Fontanine and B. Kim, {\em Moduli stacks of stable toric quasimaps, } Advances in Mathematics
225 (2010), no. 6, 3022--3051.

\bibitem{CKM} I. Ciocan-Fontanine,  B. Kim, and D. Maulik,  {\em Stable quasimap to GIT quotients,} arXiv:1106.3724.




\bibitem{CH} W. Crawley-Boevey and M. Holland, {\em Noncommutative deformations of Kleinian
singularities,} Duke Math. J. 92 (1998), 605--635.


\bibitem{CB1} W. Crawley-Boevey, {\em On the exceptional fibres of Kleinian singularities, } Amer. J. Math. 122 (2000), no. 5, 1027-1037.


\bibitem{CB2} W. Crawley-Boevey, {\em Geometry of the moment map for representations of quivers, } Compositio Math. 126 (2001), no. 3, 257-293.


\bibitem{Dia} D.E. Diaconescu, {\em Chamber structure and wallcrossing in the ADHM theory of curves I,}
              Journal of Geometry and Physics 62 (2012), 523-547.
                                

\bibitem{PG} P. Gabriel, {\em Unzerlegbare Darstellungen I,} Manuscripta Mathematica, Volume 6, Number 1 (1972), 71-103.

\bibitem{GK} P.B. Gothen and A.D. King,
{\em Homological algebra of twisted quiver bundles,} J. London Math. Soc. (2) 71 (2005), no. 1, 85--99.



\bibitem{J2} D. Joyce, {\em Configurations in abelian categories. II. Ringel-Hall algebras,} Advances in Mathematics 210 (2007), no. 2, 635--706.

\bibitem{J3} D. Joyce, {\em Configurations in abelian categories. III. Stability conditions and identities,} Advances in Mathematics 215 (2007), no. 1, 153--219.

\bibitem{JS} D. Joyce and Y. Song, {\em A theory of generalized Donaldson-Thomas invariants,}
                               Mem. Amer. Math. Soc. 217 (2012), no. 1020, iv+199.


\bibitem{K} B. Kim, {\em Stable quasimaps to holomorphic symplectic quotients,} arXiv:1005.4125.

\bibitem{KS}  M. Kontsevich and Y. Soibelman,
{\em Stability structures, motivic Donaldson-Thomas invariants and cluster transformations, } arXiv:0811.2435.


\bibitem{M} K. Miyajima, {\em Kuranishi family of vector bundles and algebraic description of the moduli space of Einstein-Hermitian connections,}
 Publ. Res. Inst. Math. Sci. 25 (1989), no. 2, 301--320.



\end{thebibliography}
\end{document}